\documentclass[leqno, 10pt]{amsart}

\usepackage{amssymb}

\theoremstyle{plain}
\newtheorem{theorem}{\indent\sc Theorem}[section]
\newtheorem{lemma}[theorem]{\indent\sc Lemma}

\theoremstyle{definition}

\newtheorem{remark}[theorem]{\indent\sc Remark}

\begin{document}

\title[On harmonic functions and the hyperbolic metric]{On harmonic functions and the hyperbolic metric}

\author[Marijan Markovi\'{c}]{Marijan Markovi\'{c}}
\address{
{Faculty of Natural Sciences and Mathematics \endgraf
University of Montenegro \endgraf
Cetinjski put b.b. 81000 Podgorica\endgraf
Montenegro}
}
\email{marijanmmarkovic@gmail.com}



\begin{abstract}
Motivated by some recent results of Kalaj and Vuorinen (Proc. Amer. Math.
Soc., 2012),    we prove that positive harmonic functions defined in the
upper half--plane are contractions w.r.t.          hyperbolic metrics of
half--plane and positive part of the real line, respectively
\end{abstract}

\maketitle

\section{Introduction and the main result}

Denote by     $\mathbb{U}=\{z\in\mathbb{C}:|z|<1\}$ the unit disc of the
complex plane $\mathbb{C}$ and by       $\mathbb{H}  = \{z\in\mathbb{C}:
{\rm Im}\, z>0\}$ the upper half--plane. $\mathbb{R}$  is the whole real
axis,  and   the positive   real axis is denoted      by $\mathbb{R}^+ =
\{x\in\mathbb{R}:x>0\}$.

Let    $d_h$ stands for the hyperbolic distance on the disc $\mathbb{U}$.
With the same letter  we denote the hyperbolic distance  on $\mathbb{H}$
and    $\mathbb{R}^+$,   since we believe that misunderstanding will not
occur. We  have
\begin{equation*}
d_h(z,w)=\inf_\gamma \int _\gamma \frac {|d\omega|}{1-|\omega|^2}
= 2 \tanh^{-1}\left|\frac{z - w}{1 - \overline{z}w}\right|.
\end{equation*}
where  $\gamma\subseteq\mathbb{U}$  is      any regular curve connecting
$z\in\mathbb{U}$ and $w\in\mathbb{U}$. On the other hand, the hyperbolic
distance between $z\in\mathbb{H}$ and $w\in\mathbb{H}$ is
\begin{equation*}
d_h(z,w)=\inf_\gamma\int_\gamma \frac {|d\omega|}{{\rm Im}\ \omega}
=2 \tanh ^{-1}\left|\frac{z-w}{\overline{z}-w}\right|,
\end{equation*}
where now  $\gamma\subseteq\mathbb{H}$. In particular,   the  hyperbolic
distance  between $x\in\mathbb{R}^+$ and $y\in\mathbb{R}^+$, where $x\le
 y$  is
\begin{equation*}
d_h(x,y)=d_h(ix,iy)= \int_x^y\frac {dt}t=\log \frac yx.
\end{equation*}

Recall the  classical Schwarz--Pick lemma. An analytic function  $f$ of
the  unit disk into itself satisfies
\begin{equation*}
\left|\frac{f(z)-f(w)}{1-\overline{f(z)}f(w)}\right|\le\left|\frac{z-w}{1-\overline{z}w}\right|
\end{equation*}
for all  $z,\, w\in\mathbb{U}$. The equality sign occurs if and only if
$f$   is a M\"{o}bius transform of $\mathbb{U}$ onto itself.

The previous      result has a counterpart  for  analytic functions $f:
\mathbb{H} \rightarrow\mathbb{H}$. Using       the Cayley transform one
easily finds that  the Schwarz--Pick inequality   in this settings says
\begin{equation}\label{SCHWARZ.PICK.HALFPLANE}
\left|\frac{f(z)-f(w)}{\overline{f(z)}-f(w)}\right|\le \left|\frac{z-w}{\overline{z}-w}\right|
\end{equation}
for every $z,\, w\in\mathbb{H}$.            Letting $z\rightarrow w$ in
\eqref{SCHWARZ.PICK.HALFPLANE},  we obtain
\begin{equation}\label{SCHWARZ.PICK.HALFPLANE.DERIVATE}
\frac{|f'(z)|}{{\rm Im}\ f(z)}\le \frac{1}{{\rm Im}\ z}.
\end{equation}
Regarding   the   expression   for the hyperbolic distance in the upper
half--plane,    \eqref{SCHWARZ.PICK.HALFPLANE}      may be rewritten as
\begin{equation}\label{SCHWARZ.PICK.HALFPLANE.DH}
d_h(f(z),f(w))\le d_h(z,w),
\end{equation}
which    means    that $f$ is a contraction in the hyperbolic metric of
$\mathbb{H}$.     It is well known that the equality  sign attains   in
\eqref{SCHWARZ.PICK.HALFPLANE}, \eqref{SCHWARZ.PICK.HALFPLANE.DERIVATE},
and \eqref{SCHWARZ.PICK.HALFPLANE.DH} (for    some     $z$ or for  some
distinct $z$ and $w$, and therefore for all such points) if and only if
\begin{equation*}
f=\text{a M\"{o}bius transform of $\mathbb{H}$ onto itself}.
\end{equation*}

During the past decade,    harmonic  mappings and  functions  have been
extensively studied and many results     from the theory   of  analytic
functions have been extended for them.

Quite    recently Kalaj and Vuorinen \cite{KALAJ.VUORINEN.PAMS}  proved
that  a harmonic        function $f:\mathbb{U}\rightarrow (-1,1)$ is  a
Lipschitz    function in the hyperbolic metric, i.e., for every $z,\, w
\in\mathbb{U}$  they  obtained that
\begin{equation}\label{LIP.DH}
d_h(u(z),u(w))\le \frac 4\pi d_h(z,w).
\end{equation}
Actually,   using        the classical results they firstly established
\begin{equation}\label{NABLA.HARMONIC}
|\nabla u(z)|\le\frac 4\pi \frac{1-|u(z)|^2}{1-|z|^2}
\end{equation}
for $z\in\mathbb{U}$. Both inequalities   are  sharp.

We refer to ~\cite{CHEN} for a related result.

We  are  interested  here  in  the  positive harmonic functions defined
in  $\mathbb{H}$. As we have said in the abstract,  our main  aim is to
prove

\begin{theorem}\label{TH}
Let $u:\mathbb{H}\rightarrow\mathbb{R}^+$ be harmonic. Then
\begin{equation}\label{CONT.DH}
d_h(u(z),u(w))\le d_h(z,w)
\end{equation}
for all $z,\, w\in\mathbb{H}$.    In   other words, a positive harmonic
function  is  a contractible    function    in  the   hyperbolic metric.

Moreover, if    the equality sign holds in \eqref{CONT.DH}     for some
pair of distinct points $z$ and $w$,  then  the function $u$ must be of
the following form
\begin{equation*}
u(z)
={\rm Im}(\text{a M\"{o}bius transform of}\ \mathbb{H}\ \text{onto}\ \mathbb{H}).
\end{equation*}
\end{theorem}

\begin{remark}
The     group of all conformal mappings of $\mathbb{H}$ onto itself is
given by
\begin{equation*}
\left\{\frac{az+b}{cz+d} : a,\, b,\, c,\, d\in\mathbb{R},\, ad-bc>0\right\}.
\end{equation*}
Thus, a          positive harmonic function $u(z)$ is extremal for the
inequality \eqref{CONT.DH} if and only if it has the form
\begin{equation*}
u(z) = k    \cdot  \mathrm{Im}z\quad\text{or}\quad u(z)  =  k \cdot P(z,t),
\end{equation*}
where $k>0$ and $t\in\mathbb{R}$; here
\begin{equation*}
P(z,t) = \frac 1\pi \frac {y}{(x-t)^2+y^2},
\end{equation*}
$z=x+iy\in\mathbb{H},\, t\in\mathbb{R}$ is the Poisson kernel  for the
upper  half--plane.
\end{remark}

\section{Proof of the result}

Let $\Omega$ be any domain in $\mathbb{C}$ (or in $\mathbb{R}$). A metric
density  $\rho$  is any continuous function  in $\Omega$ with nonnegative
values everywhere in $\Omega$. The $\rho$--length (or just a length) of a
curve $\gamma$  in $\Omega$ is  given by
\begin{equation}
\int_\gamma\, \rho(z)\, |dz|.
\end{equation}
The  $\rho$--distance (or just a distance) between  $z\in\Omega$  and $w
\in \Omega$ is
\begin{equation}
d_\rho(z,w)=\inf_\gamma\int_\gamma\rho(z)|dz|,
\end{equation}
where  $\gamma$  is a regular curve in $\Omega$  connecting point $z$ and
$w$. Of course, if $\Omega$ is an interval in $\mathbb{R}$, then we  need
not the infimum sign in the preceding expression for the distance function.

For a regular curve  $\gamma$  we   denote by $t_\gamma(\omega)$ the unit
tangent vector  at  a   point       $\omega\in\gamma$ consistent with the
orientation  of $\gamma$.

We   will    prove   the  following   lemma of somewhat general character.

\begin{lemma}\label{LE}
Let $\Omega\subseteq\mathbb{C}$  be a  domain,  $I\subseteq\mathbb{R}$ an
open interval, and   $u:\Omega\rightarrow I$   a smooth   function, i.e.,
$u\in    C^1(\Omega)$). Let $\rho$   be  a  metric   density  in $\Omega$
and let   $\tilde{\rho}$   be a metric   density  in the interval $I$. If
\begin{equation}\label{INEQ.RO}
\tilde\rho(u(\omega))|\nabla u(\omega)|\le\rho(\omega)
\end{equation}
for every $\omega\in\Omega$, then we have
\begin{equation}\label{INEQ.DRHO}
d_{\tilde \rho}(u(z),u(w))\le d_{ \rho}(z,w)
\end{equation}
for $z,\, w\in\Omega$.

If the equality sign is attained  in \eqref{INEQ.DRHO} for some pair  of
distinct  points $z$ and $w$, then equality holds in \eqref{INEQ.RO} for
some  $\omega\in\Omega$.
\end{lemma}

\begin{proof}
Let   $z$ and $w$  be distinct fixed points in $\Omega$. Without lost of
generality,   we may assume that $u(z)\le u(w)$.       Let $\gamma$ be a
regular curve    connecting $z$ and $w$.       Orient it from $z$ to $w$.
Denote
\begin{equation*}
I_0  =  \left\{\omega\in \gamma:\ \text{there exist}\ \lambda(\omega)>0\
\text{such that}\ {\nabla u(\omega)} = \lambda (\omega){t_{\gamma}(\omega)}\right\}.
\end{equation*}
This set may be decomposed  as
\begin{equation*}
I_0 = \bigcup_{n=1}^\infty I_n,
\end{equation*}
where $I_n$ are intervals in $\gamma$ (if the union if finite, we assume
that $I_n =  \emptyset $, starting from an integer $n_0$). Let $J_1=I_1$,
\begin{equation*}
J_2 = I_2 \setminus\left\{\omega\in\gamma : u(\omega)\in u(J_1)\right\},
\end{equation*}
and  by     induction for $n>1$,    let
\begin{equation*}
J_{n+1} =\left\{\omega\in I_n : u(\omega) \not\in u \left(\bigcup_{k=1}^n J_k\right)\right\}.
\end{equation*}
Denote
\begin{equation*}
J =\bigcup_{k=1}^\infty J_k.
\end{equation*}
Then    $[u(z),u(w)]\subseteq u(\gamma)$     and $|[u(z),u(w)]|=|u(J)|$.
Moreover, $u$ is  injective in $J$. Therefore,       by using inequality
\eqref{INEQ.RO}  we obtain
\begin{equation*}
\begin{split}
d_{\tilde \rho}(u(z),u(w)) & = \int_ {u(z)}^{u(w)}\, \tilde\rho(\tilde\omega)\, |d\tilde\omega|
= \int_J\, \tilde\rho(u(\omega))\, |\nabla u(\omega)|\, |d\omega|
\\&\le \int_J\, \rho(\omega)\, |d\omega| \le \int_{\gamma}\,  \rho(\omega)\, |d\omega|.
\end{split}
\end{equation*}
Since $\gamma$ is any curve, we have
\begin{equation*}
d_{\tilde \rho}(u(z),u(w))\le d_{\rho}(z,w),
\end{equation*}
what we have to prove.

The second part of this lemma follows immediately.
\end{proof}

As an  application  of the preceding lemma and the Schwarz--Pick lemma we
prove our main result here.

\begin{proof}[\indent \sc Proof of Theorem~\ref{TH}]
Let  $u:\mathbb{H}\rightarrow\mathbb{R}^+$ be a harmonic function. Denote
by $f$ an  analytic mapping in the upper half--plane  such that
\begin{equation*}
{\rm Im}\ f(z)=u(z)
\end{equation*}
for $z\in\mathbb{H}$. Then $f$ maps $\mathbb{H}$ into $\mathbb{H}$. Since
$|f'(z)|=|\nabla u(z)|,\, z\in\mathbb{H}$,  applying       the version of
Schwarz's lemma for analytic    function    in the upper half--plane, i.e.
\eqref{SCHWARZ.PICK.HALFPLANE.DERIVATE},  we obtain
\begin{equation}\label{INEQ}
\frac{|\nabla u(z)|}{u(z)} = \frac{|f'(z)|}{{\rm Im}\ f(z)}\le \frac 1{{\rm Im}\ z}
\end{equation}
for $z\in\mathbb{H}$. The equality sign appears   if and only if $f$   is
a M\"{o}bius transform of  $\mathbb{H}$ onto $\mathbb{H}$.     This is  a
counterpart   of  the sharp estimate  \eqref{NABLA.HARMONIC} for harmonic
functions in the unit disc.

Thus    $u$ satisfies the condition of Lemma \ref{LE} with the hyperbolic
metric density on each sides.     Thus, according to this lemma we obtain
\begin{equation}\label{CONT.DH.2}
d_h(u(z),u(w))\le d_h(z,w)
\end{equation}
for all $z,\, w\in\mathbb{H}$.

If  the      equality sign holds for a pair of distinct points  $z,\, w\in
\mathbb{H}$, then we must have the equality sign in \eqref{INEQ}  for some
$z$.   As we know, this   means          that $u$ must be of the form
\begin{equation*}
u ={\rm Im}(\text{a M\"{o}bius transform of $\mathbb{H}$ onto itself}),
\end{equation*}
what proves the second part of our theorem.
\end{proof}

\begin{remark}
In  order to prove that the inequality \eqref{NABLA.HARMONIC} is sharp the
authors of \cite{KALAJ.VUORINEN.PAMS} have found a function for which  the
inequality \eqref{NABLA.HARMONIC} reduces to the equality.         However,
following  the proof  in   \cite{KALAJ.VUORINEN.PAMS}, one can deduce that
the equality sign holds in  \eqref{NABLA.HARMONIC}      (for some $z$, and
therefore for all $z$) if  and only if $u(z):\mathbb{U}\rightarrow (-1,1)$
has the form
\begin{equation*}
u(z)={\rm Re}\left\{\frac {2i}\pi \log\frac{1+b(z)}{1-b(z)}\right\} = - \frac 2{\pi}{\arg} \frac{1+b(z)}{1-b(z)} ,
\end{equation*}
where $b(z)$    is    a M\"{o}bius transform of the unit disc onto itself.

Using Lemma~\ref{LE}  and the sharp estimate   \eqref{NABLA.HARMONIC} for
harmonic functions $u:\mathbb{U}\rightarrow (-1,1)$        one can derive
\eqref{LIP.DH} (what is the main result of   \cite{KALAJ.VUORINEN.PAMS}),
along with all extremal functions.
\end{remark}

\end{document}